\RequirePackage{fix-cm}
\documentclass[twocolumn]{svjour3}                     
\smartqed  
\usepackage{times}
\usepackage{epsfig}
\usepackage{amssymb}
\usepackage{enumerate}

\newcommand{\field}[1]{\mathbb{#1}}
\newcommand{\beq}{\begin{equation}}
\newcommand{\eeq}{\end{equation}}
\newcommand{\rank}{{\rm rank}}
\usepackage{graphicx}
 \journalname{Arxiv}
\begin{document}

\title{The bispectrum as a source of phase-sensitive invariants for Fourier descriptors: a group-theoretic approach\thanks{The support of AcRF Tier 1 Research Grant RG 033-09 is gratefully acknowledged.}
}

\titlerunning{Bispectral Fourier Descriptors}        

\author{Ramakrishna Kakarala 
}


\institute{R. Kakarala \at
              Nanyang Technological University \\
              Tel.: +65-6790-5925\\
              Fax: +65-6792-6559\\
              \email{ramakrishna@ntu.edu.sg}           
}
\date{This article is dedicated to the memory of Professor Bruce Michael Bennett}

\maketitle

\begin{abstract}
This paper develops the theory behind the bispectrum, a concept that is well established in statistical signal processing but not, until recently, extended to computer vision as a source of frequency-domain invariants.   Recent papers on using the bispectrum in vision show good results when the bispectrum is applied to spherical harmonic models of three-dimensional ($3$-D) shapes, in particular by improving discrimination over previously-proposed magnitude invariants, and also by allowing detection of neutral pose in human activity detection.  The bispectrum has also been formulated for vector spherical harmonics, which have been used in medical imaging for $3$-D anatomical modeling.  In a paper published in this journal, Smach {\it et al.} use duality theory to establish the completeness of second-order invariants which, as shown here, are the same as the bispectrum.  This paper unifies earlier works of various researchers by deriving the bispectrum formula for all compact groups. It also provides a constructive algorithm for recovering functions from their bispectral values on $SO(3)$.  The main theoretical result shows that the bispectrum serves as a complete source of invariants for homogeneous spaces of compact groups, including such important domains as the sphere $S^2$. 
  

\end{abstract}

\section{Introduction}
\label{intro}

The classical approach to computing invariants for pattern recognition takes place in what, in signal processing terms, is called the ``time domain''.  Moments are computed using formulae such as the following, for three-dimensions (3-D): 
\beq
m_{pq}^{\ell} = \int_{-\infty}^{\infty} \int_{-\infty}^{\infty} \int_{-\infty}^{\infty}  x^{p}y^{q}z^{\ell-p-q} f(x,y,z) dx\, dy\, dz.
\eeq
This approach, which analyzes the data expressed by $f$ directly on its domain $\field{R}^{3}$, differs fundamentally from an alternative which may be called the ``frequency domain'' approach.  The latter relies on Fourier expansion on $\field{R}^{3}$, or the equivalent harmonic expansion on different domains.  The Fourier expansion is expressed for those functions whose domain is the sphere using spherical harmonics as follows:
\beq
f(\theta,\phi) = \sum_{\ell=0}^{\infty}\sum_{k=-\ell}^{\ell} F_{\ell}^ k Y_{\ell}^{k}(\theta,\phi).
\label{eq:spharm}
\eeq
The power-spectral invariant (again using signal processing terminology) to rotation for spherical harmonic coefficients is the norm of the $(2\ell+1)$-dimensional vector 
$\|F_{\ell}\| = \sqrt{\sum_{k=-\ell}^{\ell} |F_{\ell}(k)|^2}$.  The book by Flusser \cite{flusserbook} provides a detailed review of moment invariants, which have been investigated for several decades now \cite{Hu}, and also discusses spherical harmonic invariants, which are, in comparison, recent.  

\begin{figure}[t]
\begin{center}
   \includegraphics[width=0.75\linewidth]{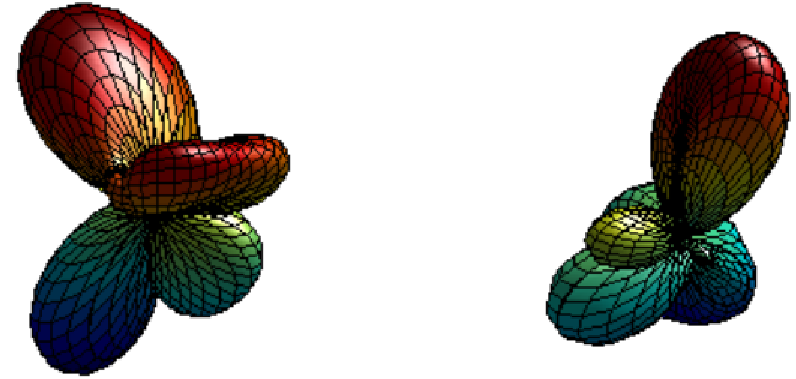}
\end{center}
   \caption{The two shapes shown are not rotations of each other, and yet have identical spherical harmonic power spectra. The
   shapes are discriminated easily by using the bispectral invariants described in this paper.}
\label{fig:onenotrot}
\end{figure}

The power spectrum invariant is relatively simple, but it lacks discriminative power.  Figure~\ref{fig:onenotrot} shows an example of how the power spectrum can be fooled into confusing dissimilar shapes. Although the power spectrum provides a weak description,  it is worth noting that there are signals that may in fact be recovered from their power spectrum alone.  For example, that is true if a signal on $\field{R}$ is positive definite, or has a known compact support with certain additional structure: Cand\`{e}s, Stromer, and Voroniskii \cite{candesphase} discuss the ``phase-retrieval'' problem in detail and provide an approach to solving it using convex programming.  The power spectrum is not the only choice if invariants are required. An alternative class of frequency-domain invariants is obtained from the signal processing concept of the {\em bispectrum}, which is explored in detail in this paper.  This concept originates from the statistical theory of polyspectra, which aims to capture statistical information that differs from an assumed Gaussian model \cite{brillinger}.  The bispectrum is the Fourier transform of the triple correlation, which is similar to the familiar autocorrelation but has one additional shift.  To illustrate with a simple example, the triple correlation for one-dimensional data is defined as 
\beq
a_{3,f}(s_1, s_2) = \int_{-\infty}^{\infty}  f^{*}(x) f(x+s_1) f(x+s_2) dx
\eeq
The bispectrum is the Fourier transform of $a_{3,f}$ and, in terms of the Fourier transform $F$ of the original function $f$, it is the following:
\beq
A_{3,f}(u_1,u_2) = F(u_1)F(u_2)F^{*}(u_1+u_2).
\label{eq:bilinear}
\eeq
It is easy to see that the bispectrum is translation-invariant, since a translation $x\mapsto x+t$ is, in the frequency domain, the transformation $F(u)\mapsto F(u)e^{jut}$, and the product of three terms on the right cancels out the linear phase component $e^{jut}$.  Moreover, when one argument is zero, for example $u_2=0$, then $A_{3,f}(u,0)=F(0)|F(u)|^2$, and hence the bispectrum contains the power spectrum $|F(u)|^2$.  

The bispectrum has been extended by several researchers from the simple linear case expressed in (\ref{eq:bilinear}) to spherical harmonics, to vector spherical harmonics, as well as to handle ranking data, in ways that are described below.  It is the purpose of this paper to bring together the different approaches in a unifying manner by describing the bispectrum using group theory, and to analyze the underlying theory using duality theory in a manner similar to the recent paper of Smach {\it et al} appearing in this journal \cite{smach}.  Though we define the terms precisely later on, let us introduce the main results.  Let $G$ be a compact group with Haar measure $dg$.  Define the triple correlation of a function $f$ on $G$ to be
\beq
a_{3,f}(g_1,g_2) = \int_{-\infty}^{\infty}  f^{*}(g) f(gg_1) f(gg_2) dg
\label{eq:triplecorre}
\eeq
We derive the bispectral invariants of the data using the Fourier transform on $G$.  The transform requires the irreducible unitary representations $D_{\sigma}$ of $G$, which are intuitively the basis elements of the Fourier transform. Here $\sigma$ denotes an index whose values lie in the {\em dual object} of $G$, which is denoted ${\cal G}$. The dual object is the set of equivalence classes of irreducible unitary representations of $G$, as defined in Hewitt \& Ross \cite[pg 2]{hewittross}. The matrix-valued Fourier coefficient for each $\sigma\in{\cal G}$ is 
\beq
F(\sigma) = \int_{G} f(g) D_{\sigma}(g)^{\dagger} dg
\label{eq:groupfour}
\eeq
The reader who is not familiar with group representation theory may consider that (\ref{eq:groupfour}) behaves in most ways like the ordinary Fourier series coefficient, which is a scalar, and which is obtained with $D_{\sigma}(\phi) = e^{j\sigma \phi}$ for integer $\sigma$ when $G$ is the periodic interval $[0,2\pi]$.    Below, we provide an example when $G$ is the 3-D rotation group $SO(3)$ below.  Using (\ref{eq:groupfour}), the bispectrum is the Fourier transform of the triple correlation and has the form
\beq
A_{3,f}(\sigma,\delta) = \left[ F(\sigma)\otimes F(\delta)\right] C_{\sigma\delta} \left[ \bigoplus_{\omega\in \sigma\otimes\delta} F(\omega)^{\dagger}\right] C_{\sigma\delta}^{\dagger}
\label{eq:bis}
\eeq
Here the notation is as follows: $\otimes$ is the Kronecker product of matrices, $\oplus$ is direct sum,  $C_{\sigma\delta}$ is the unitary matrix known as the Clebsch-Gordan matrix in studies of angular momentum \cite{hamermesh}; and $\omega\in \sigma\otimes\delta$ means all irreducible representations $\omega$ contained in the Kronecker product of representations $D_\sigma$ and $D_\delta$. The set of $\omega$ such that $\omega\in\sigma\otimes\delta$ depends on the group $G$; for example, when $G=SO(3)$, the indices $\sigma$, $\delta$ are non-negative integers, and $\omega$ takes integer values from $|\sigma-\delta|$ to $\sigma+\delta$, inclusive. Smach {\it et al} \cite{smach} introduce a ``second-order descriptor'' similar to (\ref{eq:bis}). For each $\sigma$, $\delta$ in ${\cal G}$, define  
\beq
F(\sigma\otimes \delta) = \int_{G} f(g) \left[ D_\sigma(g) \otimes D_{\delta}(g) \right]^{\dagger} dg.
\label{eq:fsigdelta}
\eeq
Then Smach {\it et al}'s descriptor is 
\beq
I^{2}_{f}(\sigma,\delta)=\left[F(\sigma)\otimes F(\delta)\right] F(\sigma\otimes \delta).
\label{eq:smach2nd}
\eeq
Below,  we compare $I^{2}$ in (\ref{eq:smach2nd}) to the bispectrum $A_{3}$ and point out the $A_{3}$ is much simpler to compute, though the two forms turn out to be mathematically equivalent.  It is important to note that the bispectrum, both for the Euclidean and for non-commutative domains, has been described by researchers in statistics, astrophysics and quantum mechanics prior to its appearance in the computer vision literature.   As shown in Section~\ref{sec:algos}, the formula (\ref{eq:bis}) allows a constructive algorithm for recovering functions from their bispectral values on $SO(3)$, unlike (\ref{eq:smach2nd}). Furthermore, the main theoretical result in this paper shows that eq.~(\ref{eq:bis}) serves as a complete source of invariants for homogeneous spaces of compact groups, including such important domains as the sphere $S^2$. 

The motivation for studying the bispectrum using group theory comes from two sources.  First, the bispectrum has been used in several fields, ranging from astrophysics,  statistics of ranking data, and, recently, in computer vision.  The reason that is has found such wide application come not only from the statistical basis that it allows testing for non-Gaussianity, but also that it detects structural aspects of---using signal processing terms---phase as opposed to magnitude.  Those different applications suggest that a unifying theory should exist, and in this paper one is proposed using group representations.  The second reason is that the bispectrum has a practical benefit for computer vision in providing a compact set of rotation invariants that improves discrimination, and detects bilateral reflection symmetry.   The group theoretic approach, specifically using duality theory, shows that such invariants have the very desirable property of completeness.  

The organization of this paper is as follows. The first half of the paper provides motivating examples of the bispectrum and its applications,  to provide context for the detailed mathematical analysis that follows in the second half. Section~\ref{sec:review} provides a review of five different applications of the bispectrum, drawn from various fields by multiple researchers, that have an underlying group theoretic basis.  Subsequently, the group-theoretic foundations of the bispectrum are described, and the next two sections provide the principal contributions of this paper: proofs of the completeness of the bispectrum, using both duality theory and constructive methods. The main result, establishing the completeness of the bispectrum for homogeneous spaces of compact groups, is provided in Section~\ref{sec:homog}. Subsequently, we illustrate applications of the theory to shape analysis and retrieval.  
This paper builds on previous publications of the author \cite{kakarala93}\cite{kakaralacvpr2010}\cite{kakaralahau3d}.  However, the theory provided in the middle sections of this paper has not been published previously\footnote{A previous version, appearing in the author's unpublished PhD thesis \cite{kakaralathesis}, has been cited by several researchers.}.

\section{Review of literature on bispectra and related concepts}
\label{sec:review}

In order to meaningfully discuss the literature on bispectral methods in computer vision, it is essential to introduce the following notation and basic concepts.  Classic books \cite{courant-hilbert}\cite{hamermesh} provide details.  For the spherical harmonic expansion of eq.~(\ref{eq:spharm}), let $F(\ell)$ denote the $2\ell+1$-dimensional row vector containing the coefficients $F_{\ell}^{-\ell}$, $\ldots$, $F_{\ell}^{\ell}$.  Let $u$ denoted the unit vector in 3-D expressed in terms of colatitude $\theta$ and longitude $\phi$ as follows
\beq
u = \left[  \cos(\phi)\sin(\theta), \sin(\phi)\sin(\theta), \cos(\theta) \right]^{T}.
\eeq
If $f$ is a function whose domain is the sphere $S^{2}$, let $f(u)$ denote the value of $f$ at the unit vector $u$, and let $F_{\ell}$, for every $\ell\geq 0$ denote its spherical harmonic coefficient vectors.   The coefficients vectors display a rotation property that is essential to our discussion.  The property is perhaps best understood by considering the group-theoretic basis.  Let $SO(3)$ denote the 3-D rotation group, which is, as is well known, a compact topological group \cite{hewittross}.   If we rotate a function $f$ on the sphere by $R\in SO(3)$, then $f(u)\mapsto f(Ru)$ for all $u$, and correspondingly the Fourier coefficient vectors undergo a unitary transformation
\beq
F_\ell \mapsto F_\ell D_\ell(R).
\label{eq:rotprop}
\eeq
The $2\ell+1$-dimensional matrices $D_{\ell}$ are unitary {\em representations} of $SO(3)$: for every $R,S$ in $SO(3)$, and with $\dagger$ denoting conjugate-transpose, we have the properties
\beq
D_{\ell}(R)^{\dagger}D_{\ell}(R) = I; \quad D_{\ell}(RS) = D_{\ell}(R)D_\ell(S).
\eeq
Furthermore, the $D$ matrices have an important product decomposition formula, which may be described as follows.  Let $\otimes$ denote the matrix tensor product: if $A$ is a $n\times n$ matrix, and $B$ an $m\times m$ matrix, then $A\otimes B$ is the $nm\times nm$-dimensional matrix obtained by repeating $B$ multiplied by $a_{ij}$ for $1\leq i,j\leq n$. Similarly $A\oplus B$ denotes the $n+m$ block diagonal matrix containing $A$ in its upper left $n\times n$ entries, $B$ in its lower right $m\times m$ entries.  With that notation, the tensor product formula decomposition on $SO(3)$ is 
\begin{eqnarray}
D_{p}(R)\otimes D_{q}(R) &=& C_{pq}\left[ D_{p+q}(R)\oplus D_{p+q-1}(R)\oplus \cdots \right.\nonumber\\
 & & \oplus D_{|p-q|}(R)] C_{pq}^{\dagger}.
\label{eq:tencompso3}
\end{eqnarray}
The matrix $C_{pq}$ is a unitary matrix that is called the {\it Clebsch-Gordan} matrix of the decomposition.  What (\ref{eq:tencompso3}) says is that the tensor product of representations for $p$ and $q$, which is a relatively large matrix, may be decomposed into a much simpler block diagonal structure: more formally, the tensor product is unitarily equivalent to a direct sum of lower dimensional representations.

With that background, we can describe five different applications of the bispectrum for compact groups that have appeared in the literature.  The survey motivates the development of a general theory, which is provided in the following sections.

\subsection{Astrophysics}

The Big Bang theory predicts an observable cosmic microwave background (CMB) radiation, which is a thermal radiation at approximately $2.7$K that fills the observable universe in an almost uniform manner.   Spectral analysis of CMB data uses spherical harmonic expansion of the sky map.   Sefusatti {\it et al.} \cite{CosmoBis06} show that the bispectral analysis improves estimation of cosmological constants over the traditionally used power spectrum.  The formulae for bispectral analysis, described earlier in astrophysics by Luo in 1994 \cite{luo}, are as follows.  Using the expansion (\ref{eq:spharm}), the power spectrum is defined by 
\beq
P(\ell) = F_{\ell} F_{\ell}^{\dagger} = \sum_{k=-\ell}^{\ell}\left| F_{\ell}^{k} \right|^2.
\label{eq:power}
\eeq
The bispectrum is defined by Luo as the coefficients, denoted $B_3$, in the Fourier expansion of the three-point correlation function:
\begin{eqnarray}
 E\left[ f(u_1)f(u_2)f(u_3)\right] &=& \sum_{\ell_i,k_i} B_3(\ell_1,k_1,\ell_2,k_2,\ell_3,k_3)\nonumber\\
 & &Y_{\ell_1}^{k_1}(u_1)Y_{\ell_2}^{k_2}(u_2)Y_{\ell_3}^{k_3}(u_3).
 \label{eq:bispecluo}
\end{eqnarray}
Luo shows that $B_3=0$ unless the following conditions are met: $k_1+k_2+k_3=0$; for all $i$, we have $\ell_i \leq | \ell_j - \ell_k |$; and $\ell_1 + \ell_2 + \ell_3$ is even.  We may interpret $B_3$ as the Fourier expansion of the stochastic triple correlation, the deterministic version of which is shown above (\ref{eq:triplecorre}).  Consequently, the coefficients $B_{3}$ on the right side of (\ref{eq:bispecluo}) will match in form the expression (\ref{eq:bis}).  

\subsection{Ranking data}

Ranking data arise in search engine rankings, surveys and elections.  A ranking is an ordering of the numbers $\{1,2,\ldots,n\}$.  The $n!$ possible orderings form the discrete symmetric group, which is denoted $S_n$.  We may use a function $f:S_n\rightarrow \field{N}$ to express the results of a survey where respondents are asked to rank $n$ items; here, for each $p\in S_n$, the value of $f(p)$ is the number of times that ordering $p$ is chosen by survey respondents.  Diaconis \cite{diaconismono} shows the value in analyzing ranking data in the frequency domain.  Because $S_n$ is a compact group, the $\ell$-th Fourier coefficient of $f$ is obtained by the integral
\beq
F(\ell) = \sum_{p\in S_n} f(p) U_{\ell}(p)^{\dagger}.
\label{eq:sncoeff}
\eeq
Here $U_{\ell}$ is the Fourier basis obtained from, for example, the Young orthogonal representation of the symmetric group; see \cite{diaconismono} for details.  The Fourier coefficients capture the variation in the ranking data in terms of the modes at which it occurs.  Kakarala \cite{kakaralatsp} shows that Fourier analysis of ranking data breaks down into magnitude and phase, with the phase information conveying relative position.  Furthermore, it is shown in that paper that several real-world ranking datasets have linear phase, which means that they are symmetric with respect to inversion.  Linear phase may be detected with the aid of the bispectrum, defined in \cite{kakaralatsp} as 
\beq
B(\sigma,\delta) = \left[ F(\sigma)\otimes F(\delta)\right] C_{\sigma\delta} \left[\bigoplus_{\omega\in\sigma\otimes\delta} F(\omega)^{\dagger}\right] C_{\sigma\delta}^{\dagger}.
\label{eq:bispectum}
\eeq
This is of course exactly the same as $A_{3,f}$ defined in (\ref{eq:bis}).

\subsection{Computer vision}

Researchers in computer vision and medical imaging have shown considerable interest in modeling 3-D shape using spherical harmonics.  Kazhdan {\em et al} \cite{KazhdanFR03} show that the power spectrum (\ref{eq:power}) provides a compact yet discriminative set of rotation invariants.  Kakarala \& Mao \cite{kakaralacvpr2010} showed that bispectral invariants are superior to the power spectral invariants of \cite{KazhdanFR03} in discrimination of objects.  Their results are discussed further below. The bispectral invariants of \cite{kakaralacvpr2010} are defined as follows:
\beq
b_{f}^{k\ell}(i) = \left[F_k \otimes  F_\ell\right] C_{k\ell} \check{F}_{i}^{\dagger}; \quad |k-\ell| \leq i \leq k+\ell.
\label{eq:redbis}
\eeq
Here $C_{k\ell}$ is the $(2k+1)(2\ell+1)$-dimensional  unitary Clebsch-Gordan matrix appearing in (\ref{eq:tencompso3}), and the symbol $\check{F}_i$ denotes the $2i+1$-dimensional vector of spherical harmonic coefficients $F_{i}$ padded with zeros to be $1\times(2k+1)(2\ell+1)$ in size, so as to complete the weighted inner product with the matrix $C_{k\ell}$. The reason for the zero-padding is that the direct sums in (\ref{eq:bis}), which stack matrices in a block-diagonal manner, leave many zeros on either size of the coefficient vectors. This is described further in an earlier paper \cite{kakaralacvpr2010}.

Reisert \& Burkhardt \cite{ReisertB06} proposes to use the group representations on $SO(3)$ to form invariants.  Specifically, given a 3-D object $X$ and a scalar function $f$ operating on $X$, let the $\ell$-th projection be the $(2\ell+1)$-dimensional matrix
\beq
F(\ell) = \int_{SO(3)} f(RX) D_{\ell}(R)^{\dagger} dR
\label{eq:project}
\eeq
Here $dR$ is the rotation invariant Haar measure on $SO(3)$.    Under 3-D rotation by an element $Q$ of $SO(3)$, the coefficient matrix $F(\ell)$ undergoes the transformation 
$F(\ell)\mapsto F(\ell)D_{\ell}(Q)$.  Because $D_{\ell}$ is unitary, Reisert proposes to use the norms of the rows of $F(\ell)$ which remain invariant under the unitary transformation.  Essentially, this invariant is of power-spectral type as it uses the positive semidefinite matrix $F(\ell)F(\ell)^{\dagger}$, the diagonal entries of which are respective norms of the rows of $F(\ell)$.  

Chung \cite{chung} proposes to use the vectorial harmonics known as SPHARM, which are essentially spherical harmonic expansions for each of the coordinates of a function $V:S^2\rightarrow \field{R}^{3}$ for modeling cortical surfaces.  Fehr \cite{Fehr10}  extends the power-spectral and bispectral invariants to SPHARM.  The formulation may be described as follows.  Let $V(\theta,\phi) = \left[V_{1}(\theta,\phi), V_{2}(\theta,\phi), V_{3}(\theta,\phi)\right]^{\top}$ denote the vector-valued function $V$ expressed in terms of its three components as scalar functions on $S^2$.  Then, from (\ref{eq:spharm}), the $k$-th component function has the spherical harmonic expansion
\beq
V_{k}(\theta,\phi) = \sum_{\ell=0}^{\infty} \sum_{n=-\ell}^{\ell} F_{k,\ell}^{n} Y_{\ell}^{n}(\theta,\phi).
\label{eq:kthcomp}
\eeq
For each $\ell$, the coefficients $\{F_{k,\ell}^{n}\}$ may be assembled into $3\times (2\ell+1)$ matrix denoted ${\cal F}_{\ell}$, whose $k$-th row is the vector $F_{k,\ell} = \left[F_{k,\ell}^{-\ell},\ldots,F_{k,\ell}^{\ell}\right]$. If we rotate the $k$-th component function on $S^2$ by $V_{k}(u)\rightarrow V_k(Ru)$ for $R\in SO(3)$ and $u\in S^2$, then, from (\ref{eq:rotprop}), we have that 
$F_{k,\ell} \rightarrow F_{k,\ell} D_{\ell}(R)$. Consequently, if the same rotation $R$ is applied to all three components of $V$, we obtain ${\cal F}_\ell \rightarrow {\cal F}_{\ell} D_{\ell}(R)$. Therefore, for each $\ell$, the power spectrum forms a $3\times 3$ matrix of rotation invariants as follows:
\beq
{\cal P}_\ell = {\cal F}_\ell {\cal F}_\ell^{\dagger}.
\eeq
Fehr \cite{Fehr10} also formulates a bispectral invariant for SPHARM, as follows. For non-negative integers $\sigma$, $\delta$, we have 
\beq
{\cal B}_{\sigma\delta} = \left[ {\cal F}_{\sigma}\otimes {\cal F}_{\delta} \right] C_{\sigma\delta} \left[  \bigoplus_{\omega=|\sigma-\delta|}^{\sigma+\delta} \check{\cal F}_\omega^{\dagger} \right] C_{\sigma\delta}^{\dagger}.
\label{eq:fehr}
\eeq
The symbols used on the right hand side have the following meaning. The matrix $C_{\sigma\delta}$ is the unitary Clebsch-Gordan matrix shown in eq. (\ref{eq:tencompso3}).   Its dimension is $(2\sigma + 1)(2\delta+1)$ on a side.  The matrices $\check{\cal F}_\omega^{\dagger}$ are square with dimension $2\omega+1$ for each $\omega$, and contain respectively the ${\cal F}_\omega$ matrices embedded as the middle $3$ rows with the remaining rows being zeros.  We see in (\ref{eq:fehr}) another example of the bispectral formula (\ref{eq:bis}).  

\subsubsection{Generalized Fourier Descriptors}

Smach {\it et al} \cite{smach} describe a novel set of invariants called ``generalized Fourier descriptors''.  We can state those in terms of the group-theoretic Fourier transform (\ref{eq:groupfour}) as follows.  The first and second order descriptors are, respectively,

\begin{eqnarray}
&I^{1}_{f}(\sigma)=F(\sigma)F(\sigma)^{\dagger}&\\
&I^{2}_{f}(\sigma,\delta)=\left[F(\sigma)\otimes F(\delta)\right] F(\sigma\otimes \delta)&
\label{eq:smachdescriptors}
\end{eqnarray}

The first order descriptor is clearly the power spectrum.   The second order descriptor in \cite{smach} requires computation of the matrix Fourier coefficients at $\sigma$, $\delta$, and at the product representation $\sigma \otimes \delta$ as defined in (\ref{eq:fsigdelta}).  Smach {\it et al} do not discuss this point, but computing $F(\sigma\otimes\delta)$ is theoretically redundant because the representation $D_{\sigma}\otimes D_\delta$ is reducible into smaller dimensional representations.  Specifically, using the tensor product decomposition formula, we have that $\sigma\otimes \delta$ reduces into a direct sum of lower dimensional representations using the Clebsch-Gordan decomposition:
\beq
D_{\sigma}(g) \otimes D_{\delta}(g) = C_{\sigma\delta}\left[ \bigoplus_{\omega\in\sigma\otimes\delta} D_{\omega}(g) \right] C_{\sigma\delta}^{\dagger}
\label{eq:tensorprod}
\eeq
A specific example illustrates the redundancy in computing the tensor product in (\ref{eq:smachdescriptors}).  When $G=SO(3)$, the representations are indexed by integer $\ell \geq 0$ and each $D_{\ell}$ has dimension $2\ell+1$.  Consequently, we have equation (\ref{eq:tencompso3}), which shows that 
\beq
F(\sigma\otimes \delta) = C_{\sigma\delta} \left[ F(\sigma+\delta)\oplus\cdots\oplus F(|\sigma-\delta|)\right]C_{\sigma\delta}^{\dagger}.
\eeq
Consequently, if the normal Fourier coefficients $F(\ell)$ are computed for $\ell\leq \sigma+\delta$, then there is no need to compute the completely different coefficient $F(\sigma\otimes \delta)$.  Furthermore, we obtain that the bispectrum in (\ref{eq:bis}) is, on $SO(3)$, the following:
\begin{eqnarray}
A_{3,f}(\sigma,\delta) &=& \left[ F(\sigma)\otimes F(\delta) \right] C_{\sigma\delta} \left[ F(\sigma+\delta)^{\dagger}\oplus\nonumber\right.\\
& &\cdots\oplus F(|\sigma-\delta|)^{\dagger}]C_{\sigma\delta}^{\dagger}.
\label{eq:biss03}
\end{eqnarray}
The bispectrum is much simpler to work with on compact groups than (\ref{eq:smachdescriptors}) because all of the coefficients are computed using the standard Fourier basis obtained from the $D_\ell$ representations.

\section{Experimental results}

It is helpful to motivate the abstract analysis that follows with examples of applications.  The paper of Kakarala \& Mao \cite{kakaralacvpr2010} shows how bispectral invariants for spherical harmonics are derived from the group-theoretical form on $SO(3)$.   We review some of their findings here.  Specifically, to each function $f$ defined on the sphere $S^2$, which represents for example shape information or, alternatively, radiation measurements in astrophysics,  there corresponds a function $\widetilde{f}$ obtained by using the Euler angles $\alpha$, $\beta$, and $\gamma$, with $0\leq \alpha,\gamma \leq 2 \pi$ and $0\leq \beta \leq \pi$ and setting 
\beq
\widetilde{f}(\alpha,\beta,\gamma) = f(\alpha,\beta).
\eeq
Since one angle in the arguments to $\widetilde{f}$ is irrelevant, we may expect that the Fourier coefficients of $\tilde{f}$ on $SO(3)$ have a simple structure.  In particular, it is shown \cite{kakaralacvpr2010} that, if $\widetilde{F}(\ell)^{mn}$ denotes the $mn$-th element of the matrix on the left side of  eq.(\ref{eq:groupfour}), then only the middle row of the matrix is non zero, and furthermore that middle row is related to the spherical harmonic vector $F_{\ell}$ in (\ref{eq:spharm}) by a constant
\beq
m\neq 0 \Rightarrow \widetilde{F}(\ell)^{mn} = 0, \quad \widetilde{F}(\ell)^{0n} = a_\ell F_{\ell}.
\eeq
Consequently, because only the middle rows of $F(\ell)$-coefficient matrices matter, the bispectrum (\ref{eq:bis}) simplifies into the form expressed in (\ref{eq:redbis}).  Furthermore, the scalars $b_{f}^{k\ell}(i)$ have the following properties \cite{kakaralacvpr2010}:
\begin{enumerate}
\item The bispectrum contains the power spectrum: 
\[
b_{f}^{0\ell} = F_0 F_{\ell} F_{\ell}^{\dagger} = F_{0} \| F_{\ell} \|^2
\]
\item The bispectrum is invariant to rotation: if for some rotation $R$ we have $f_1(u) = f_2(Ru)$ for all $u\in S^{2}$, then the bispectra of function $f_1$, $f_2$ match, i.e., $b_{f_1} = b_{f_2}$.
\item Reflection of each real-valued function $f$ across any plane results in complex-conjugation of $b_{f}$, i.e., $f\rightarrow f^{O}$ implies that $b_{f} \rightarrow b_{f}^{*}$, where $f^{O}$ denotes the reflected function $f^{O}(u) = f(Ou)$, with $O$ a $3\times 3$ reflection matrix having ${\rm det}(O) = -1$.  Specifically, as shown in \cite[Thm 2.1]{kakaralahau3d}, bispectral terms $b_{f}^{k\ell}(i)$ with even order ($\ell + k + i$ even) are invariant under reflection because they are purely real, while odd-order bispectral terms are purely imaginary and therefore change sign.    
\end{enumerate}

\begin{figure}[t]
\begin{center}
   \includegraphics[width=0.5\linewidth]{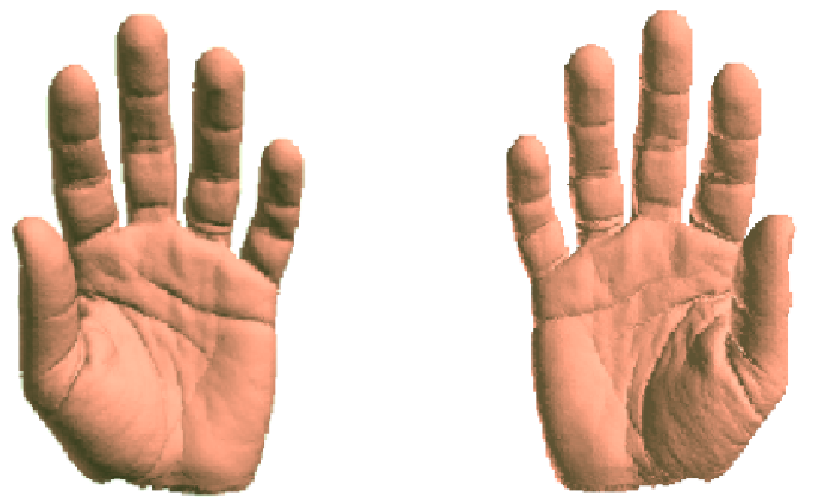}
\end{center}
   \caption{The vertices in the right hand are obtained by reflecting those of the left.  Consequently, the bispectral invariants of the right hand are complex conjugates of the left.}
\label{fig:handreflect}
\end{figure}

The last property shows that the bispectrum is able to differentiate rotation from reflection while remaining rotation-invariant. We illustrate the property using an experiment described in \cite{kakaralacvpr2010} that uses the hand shape\footnote{The hand shape appears courtesy of INRIA, and is available at the repository 
{\tt shapes.aimatshape.net}.} shown in Figure ~\ref{fig:handreflect}.  The details of computation, including how the Clebsch-Gordan coefficients are computed, are discussed in that paper.     The spherical harmonic coefficient vectors for $\ell = 1,2,\ldots,5$ are computed using the hand vertices, where on each trial the vertices were reflected across a randomly-chosen plane in 
$\field{R}^3$. The  magnitude-based invariants $\|F_{\ell}\|$ are essentially equal (to machine precision) for both the data and its reflection, despite there being no rotation to bring the two hands into alignment with one another. However, the bispectral invariants $b_{f}^{k\ell}(i)$ for $k=3$, $\ell=2$, and $i=1,\ldots,5$, are quite different. We define the normalized difference $E_{\rm odd}$ over $\ell+k+i$ odd (why only odd sum of indices is explained above) by  
\beq
E_{\rm odd} = \frac{\|b_{f} - b_{g}\|_{o}}{\|b_f\|_{o} + \| b_{g}\|_{o}}\times 100.
\label{eq:norme}
\eeq
The norm $\| \cdot \|_{o}$ is computed only over $\ell+k+i$ odd.  For the two hands, $E_{{\rm odd}}$ is nearly $100\%$. This nearly perfect difference clearly indicates the ability of the bispectrum to distinguish reflections from rotations.  Moreover, when the bispectral values for the right hand were conjugated, they matched those of the left with essentially zero error $E_{{\rm odd}}$, confirming that reflection is conjugation as stated in Property 3 above.  The experiment, when repeated with randomly chosen reflections, shows no significant change in results.   

The reflection property may also be used to determine bilateral symmetry. This application is also discussed in \cite{kakaralacvpr2010}, and is briefly described here.  If a $3$-D shape is bilaterally-symmetric, and therefore invariant to reflection about some plane, then its bispectral coefficients in (\ref{eq:redbis}) are real-valued.  To test the reliability of this property in noise, $39$ shapes were selected from the Princeton Shape Benchmark (PSB)\cite{Shilane04} which were not bilaterally symmetric, such as human hands, walking figures, and potted plants.  (Most shapes in PSB are bilaterally symmetric, because they represent balanced objects like aircraft or chairs). The $39$ shapes chosen exhibit varying levels of asymmetry. Each shape was normalized in size to $1.0$ mean distance from centroid, so that added noise may be expressed as a percentage. The shapes were reflected across a randomly-chosen plane, with Gaussian noise added to the coordinates. The distance (\ref{eq:norme}) was measured between the noise-free bispectral coefficients $b_{f}$, those of the noisy reflected shape, denoted $\hat{b}_{f}$, and their complex-conjugates $\hat{b}_{f}^{*}$.   If the bispectral coefficients $b_{f}$ are closer to $\hat{b}_{f}$ than to $\hat{b}_{f}^{*}$, then the shape is considered to be  confused with its reflection. With $5\%$ noise, the confusion rate was $2.6\%$; with $10\%$ noise, the rate was $5.1\%$; and$15\%$ noise resulted in confusion of $10.3\%$.  Hence adding up to $10\%$ noise does not incur significant error.

\begin{table}
\caption {Recognition rate (RR) of the power spectral invariants compared to the bispectrum.  We see in each case an improvement obtained by using of the bispectrum .}
\begin{center}
\begin{tabular}{|c|c|c|c|} \hline
Bandwidth $L$ & \% Noise added & RR (power spectral) & RR (bispectral)  \\
\hline 
7 & 5 & 90 & 96 \\ 
\hline
7 & 10 & 75 & 86 \\
\hline
9 & 5 & 91 & 95 \\
\hline
9 & 10 & 77 & 82 \\
 \hline
\end{tabular}
\end{center}
\label{tab:RRvsnoise}
\end{table}

The discriminative power of the bispectral invariants for arbitrary shapes may be tested using the full collection  in the PSB  \cite{Shilane04}.  The experiment is reported in \cite{kakaralacvpr2010} and described here.  Using the $907$ shapes in the PSB "testing" set, each randomly rotated with added noise over trials, we compute the recognition rate (RR), which is how often the correct model is identified in noisy data.  Table~\ref{tab:RRvsnoise} summarizes the results using invariants for $\ell\leq L$, where $L$ denotes the bandwidth.  It is clear from the table that bispectral invariants improve RR by up to $11\%$ over power-spectral invariants.  
  
\section{Theoretical results}

The main results in this section date from 1992 as they were presented in the author's unpublished PhD thesis \cite{kakaralathesis}.   The interested reader may refer to a technical report \cite{kakaralaarxiv} to see the original proofs.  Since Smach {\em et al} \cite{smach} have recently published some similar results, we emphasize only what is different from \cite{smach} in this paper.   The background material on group representations may be found in either source \cite{kakaralaarxiv}\cite{smach}.  Throughout this section $G$ denotes a compact topological group with Haar measure $dg$.  

Our first result shows that the bispectrum is the Fourier transform of the triple correlation.  This fact is important for connecting the bispectrum to the statistical basis which spawned the concept of polyspectra in the first place \cite{brillinger}.   For $f\in L_1(G)$, the triple correlation $a_{3,f}$ is formally defined in eq.(\ref{eq:triplecorre}).  The bispectrum is, as noted above, the Fourier transform of $a_{3,f}$ on $G\times G$.  Let ${\cal G}$ denote the dual object of $G$ as defined in the Introduction. Consider the tensor product $D_{\sigma}\otimes D_\delta$ of any
two representations $\sigma$, $\delta$ in ${\cal G}$; that representation is, in general, reducible, and we write its decomposition
into irreducibles as follows:
\begin{equation}
D_{\sigma}\otimes D_{\delta} = C_{\sigma\delta}\left[ D_{\alpha_1} \oplus \cdots \oplus D_{\alpha_k}\right] C_{\sigma\delta}^{\dagger}.
\label{eq:3p5}
\end{equation}
 Using this result, we show that the bispectrum may be determined entirely in terms of the Fourier transform coefficients of the underlying function
\begin{lemma}
\label{lem:3p2p3}
For $\sigma$, $\delta$ in ${\cal G}$, there is a unitary matrix $C_{\sigma\delta}$ such that 
\begin{eqnarray}
\label{eq:3p7}
A_{3,f}(\sigma,\delta) &=& \left[F(\sigma)\otimes F(\delta)\right]C_{\sigma\delta}
\left[F(\alpha_1)^{\dagger}\oplus\cdots \right. \nonumber\\ 
& & \oplus F(\alpha_k)^{\dagger}]\,C_{\sigma\delta}^{\dagger}.
\end{eqnarray}
\end{lemma}

\begin{proof}
Since $a_{3,f}$ is integrable, we use the Fubini theorem to interchange the order of integration in the following
derivation:
\begin{eqnarray*}
A_{3,f}(\sigma,\delta) &=& \int_{G}\int_{G}a_{3,f}(g_1,g_2)\\
& &\left[D_{\sigma}(g_1)^{\dagger}\otimes D_{\sigma}(g_2)^{\dagger}\right]dg_1\,dg_2,\\
&=& \int_{G}\int_{G}\int_{G} f(g)^{*} f(gg_1) f(gg_2) \\
& &\left[D_{\sigma}(g_1)^{\dagger}\otimes D_{\delta}(g_2)^{\dagger}\right] dg\,dg_1\,dg_2,\\
&=& \int_{G} f(g)^{*} \int_{G}\int_{G} f(gg_1) f(gg_2) \\
& & \left[D_{\sigma}(g_1)^{\dagger}\otimes D_{\delta}(g_2)^{\dagger}\right] dg_1\,dg_2\,dg.
\end{eqnarray*}
By making a change of variables, we find that the double integral inside simplifies as follows:
\begin{eqnarray*}
\int_{G}\int_{G}f(gg_1)f(gg_2)\left[D_{\sigma}(g_1)^{\dagger}\otimes D_{\delta}(g_2)^{\dagger}\right]dg_1\,dg_2 \\
= \left[F(\sigma)\otimes F(\delta)\right]\left[D_{\sigma}(g)\otimes D_{\delta}(g)\right].
\end{eqnarray*}
Upon substituting into the expression for $A_{3,f}$, we find that 
\[ 
A_{3,f}(\sigma,\delta) = \left[F(\sigma)\otimes F(\delta)\right] \int_{G} f(g)^{*} \left[D_{\sigma}(g)\otimes D_{\delta}(g)\right]dg.
\label{eq:beforecg}
\] 
Upon substituting the tensor product decomposition (\ref{eq:3p5}) into the above, we obtain that 
\begin{eqnarray}
A_{3,f}(\sigma,\delta) &=& \left[F(\sigma)\otimes F(\delta)\right] C_{\sigma\delta} \left[\int_{G}f(g)^{*}(D_{\alpha_1}(g) \oplus \cdots \right. \nonumber\\
& &\oplus D_{\alpha_k}(g))dg]C_{\sigma\delta}^{\dagger}.
\label{eq:aftercg}
\end{eqnarray}
After evaluating the integral, the result (\ref{eq:3p7}) follows. \end{proof}

The proof above also shows that the bispectrum is the same as the second order invariant of Smach {\it et al}.
\begin{lemma}
\label{lem:bisequali2}
$A_{3,f}(\sigma,\delta) = I^{2}_{f}(\sigma,\delta)$.
\end{lemma}
\begin{proof}
Eqns (\ref{eq:beforecg}) and (\ref{eq:aftercg}) show that 
\beq
F(\sigma\otimes\delta) =  C_{\sigma\delta}\left[ F(\alpha_1)\oplus \cdots \oplus F(\alpha_k) \right]C_{\sigma\delta}^{\dagger}.
\label{eq:Fourerequiv}
\eeq
Substiting this results into eq (\ref{eq:smachdescriptors}), we obtain (\ref{eq:3p7}).
\end{proof}

The lemma helps us to establish a fundamental theorem on the completeness on the bispectral invariants, by using the similar result of Smach {\it et al}~\cite{smach}, which is stated below. On compact groups, the Fourier coefficients are matrices, and therefore we use linear algebra terminology such as ``nonsingular'' to indicate that the matrices are invertible.   

\begin{theorem}[Thm 5, Smach {\it et al}] 
\label{thm:smach}
Let $r$, $s$ be two functions on a compact group $G$ such that the Fourier coefficient $R(\sigma)$ are nonsingular for all $\sigma$.  Then $I^{2}(r)=I^{2}(s)$ if and only if there exists $x\in G$ such that $s(g)=r(xg)$ for all $g$.

\end{theorem}

 The same result applies to the bispectrum by Lemma~\ref{lem:bisequali2}.  

\begin{theorem}
\label{thm:3p2p4}
Let $G$ be any compact group, and let $r$ in $L_{1}(G)$ be such that its Fourier coefficients $R(\alpha)$ are
nonsingular for all $\alpha\in{\cal G}$. Then $a_{3,s} = a_{3,r}$ for some $s\in L_{1}(G)$ if and only if there
exists $x\in G$ such that $s(g) = r(xg)$ for all $g$.  
\end{theorem}

\begin{proof}
If $s(g) = r(xg)$, then the translation-invariance of the triple correlation implies that $a_{3,r} = a_{3,s}$. We now prove the converse.  Let $s$ be such that $a_{3,s} = a_{3,r}$; then $A_{3,r} = A_{3,s}$, and by Lemma~\ref{lem:bisequali2}, we obtain that $I^{2}_{r} = I^{2}_{s}$ and hence by Thm~\ref{thm:smach}, we have that $s(g) = r(xg)$ for all $g$. \qquad\end{proof}

It is worth noting that a completely different proof of Theorem~\ref{thm:3p2p4}, first presented in 1992, well before the elegant proof given in \cite{smach}, is provided in the companion report \cite{kakaralaarxiv}.
The hypothesis that all coefficients $R(\sigma)$ are nonsingular is satisfied generically, in the sense that almost every $n\times n$
matrix is nonsingular with respect to the Lebesgue measure on the set of $n\times n$ matrices.  Nevertheless, it is desirable
to weaken the hypothesis, to include for example functions on $G$ that are invariant under the translations of a normal
subgroup $N$ of $G$. A companion report \cite{kakaralaarxiv} gives the proof of the following theorem.  

\begin{theorem}
\label{thm:3p2p5}
Let $r \in L_{2}(G)$ be such that its Fourier coefficients $R$ satisfy the following conditions:
\begin{enumerate}
\item Each $R(\alpha)$ is either zero or nonsingular;
\item The set of $\alpha$ such that $R(\alpha)$ is non-singular includes the trivial representation $D(g)=1$, and is closed under
conjugation and tensor product decomposition.
\end{enumerate}
Then there exists a normal subgroup $N$ of $G$ such that $r$ is $N$-invariant, and furthermore $r$ is 
uniquely determined up to left translation by its bispectrum $A_{3,f}$.
\end{theorem}

\subsection{Homogeneous spaces}
\label{sec:homog}

The definition of a homogeneous space is as follows.  Let $G$ be any topological group and 
$X$ any topological space.  We say that $G$ {\em acts} on $X$ if for each $g \in G$ there exists a
homeomorphism $\rho_{g} : X \rightarrow X$, such that $\rho_{e}(x) = x$ for the identity $e$ in $G$,
and furthermore, for $g_1$, $g_2$ in $G$, we have $\rho_{g_1 g_2}(x) = \rho_{g_1}\left(\rho_{g_2}(x)\right)$.
The group $G$ acts {\rm transitively} on $X$ if for each $x_1$, $x_2$ in $X$, there exists $g\in G$
such that $\rho_{g}(x_1) = x_2$.  The space $X$ is a {\em homogeneous space} for $G$ if $G$ acts on
$X$ transitively and continuously.  A general example of a homogeneous space is the quotient space of right cosets 
$G\backslash H = \left\{ Hg : g\in G\right\}$ of a closed subgroup $H$ in $G$. In fact, it is a theorem that virtually all homogeneous spaces for compact groups $G$ are quotient spaces of a subgroup $H$ \cite[pg 124] {barut}. A specific example is when $G$ is the 3-D rotation group $SO(3)$, and $H$ is the subgroup that fixes the North pole $z=[0,0,1]$ on the sphere; then $G\backslash H$ corresponds to the sphere $S^2$. Our goal in this section is to investigate the bispectrum's completeness for functions on
arbitrary homogeneous spaces of compact groups. By the result cited above, we lose no generality 
by focusing on spaces of the form $G\backslash H$, where $G$ is some compact group and $H$ some closed subgroup
of $G$.  To any function $\widetilde{f}$ on $G\backslash H$ there corresponds a unique function $f$ on $G$ 
such that $f = \widetilde{f}\circ\pi$, where $\pi: G \rightarrow G\backslash H$ is the canonical coset map; conversely, 
to any function $f$ on $G$ that is invariant under left $H$-translations, i.e., $f(hg)=f(g)$ for all $g\in G$
and $h\in H$, there corresponds a unique function $\widetilde{f}$ on $G\backslash H$ such that $f = \widetilde{f}\circ \pi$. Thus
we lose no generality by further restricting our study of functions on homogeneous spaces to functions on
$G$ that are left $H$-invariant for some closed subgroup $H$.  

In this subsection, Theorem~\ref{thm:3p2p4} is extended using the Iwahori-Sugiura duality theorem for homogeneous
spaces of compact groups \cite{iwahorisugiura}.  This form of duality is not covered in the cases discussed by Smach {\it et al}\cite{smach}. Let $G$ be any compact group, $\{D_{\alpha}\}_{\alpha\in{\cal G}}$ be any selection of irreducible representations. Let $\Theta(G)$ be the algebra\footnote{An algebra of functions is a set of functions closed under addition and multiplication, where the multiplication distributes over addition and is compatible with scalar multiplication. In this case, addition and multiplication of two functions $f_1$, $f_2$ are respectively $f_1(g)+f_2(g)$ and $f_1(g)f_2(g)$ for all $g\in G$.}  of complex-valued functions on $G$ generated by finite linear combinations of member functions $d_{\alpha}^{ij}(g)$ of the selection.  For
any closed subgroup $H$ of $G$, let $\Theta_{H}(G)$ denote the subalgebra of $\Theta(G)$ consisting of functions
that are invariant under left $H$-translations.  For each $f\in\Theta_{H}(G)$, let $f(Hg)$ denote the
common value given to elements of the coset $Hg$ by $f$.  The algebraic structure of $\Theta_{H}(G)$ is
revealed to a large extent by the multiplicative linear functionals $\omega: \Theta_{H}(G) \rightarrow \field{C}$,
i.e., algebra homomorphisms of $\Theta_{H}(G)$.  The Iwahori-Sugiura theorem characterizes those 
algebra homomorphisms that preserve conjugation.

\begin{theorem}[Iwahori-Sugiura] To each algebra homomorphism $\omega: \Theta_{H}(G) \rightarrow \field{C}$ that
preserves conjugation, there corresponds a unique coset $Hg$ in the quotient space $G\backslash H$ such that for all
$f \in \Theta_{H}(G)$,
\begin{equation}
\eta(f) = f(Hg).
\end{equation} 
\end{theorem}

We describe an equivalent formulation of the Iwahori-Sugiura theorem that is advantageous for our work.  
Several preliminary results are required for the new formulation, with some of the longer proofs being relegated
to the companion report \cite{kakaralaarxiv}.  

\begin{lemma}
\label{lem:3p3p2}
Any function $f\in\Theta_{H}(G)$ can be expressed as a unique finite linear combination of the left
$H$-invariant matrix coefficients of a given selection.  
\end{lemma}

The proof is given in the report \cite{kakaralaarxiv}.  

Let $G$, $H$, and $\{D_{\alpha}\}_{\alpha\in{\cal G}}$ be as before.  Let us define a corresponding sequence of 
matrices $\{P_{\alpha}\}_{\alpha\in{\cal G}}$ as follows:
\begin{equation}
P_{\alpha} = \int_{H} D_{\alpha}(h) dh,
\label{eq:3p10}
\end{equation}
where $dh$ denotes the normalized Haar measure on $H$.  It is easy to show that each $P_{\alpha}$ is a projection,
i.e., a self-adjoint matrix such that $P_{\alpha}P_{\alpha} = P_{\alpha}$ (\cite[pg 190]{hewittross}). Moreover, the 
projection matrices as defined above inherit some of the tensor product properties of the corresponding 
representations (\cite[pg 190]{hewittross}).
\begin{lemma}
\label{lem:3p3p3}
Let $\{P_{\alpha}\}_{\alpha\in{\cal G}}$ be as above.  For each $\sigma$, $\delta$, let $C_{\sigma\delta}$ be the 
Clebsch-Gordan matrix and $\alpha_1$, $\ldots$, $\alpha_k$ be the indices in the tensor product decomposition in
eq.~(\ref{eq:3p5}). Then
\begin{eqnarray*}
P_{\sigma}\otimes P_{\delta} &=& C_{\sigma\delta}\left[P_{\alpha_1}\oplus\cdots\oplus P_{\alpha_k}\right] C_{\sigma\delta}^{\dagger}
\left[P_{\sigma}\otimes P_{\delta}\right],\\
&=& \left[P_\sigma \otimes P_\delta\right]C_{\sigma\delta}\left[P_{\alpha_1} \oplus \cdots \oplus P_{\alpha_k}\right]C_{\sigma\delta}^{\dagger}.
\end{eqnarray*}
\end{lemma}
It proves convenient to apply the following similarity transformations to the $P$ matrices.  For each $\alpha$,
let ${\rm rank}(\alpha)$ denote the rank of $P_{\alpha}$, and let $I({\rm rank}(\alpha))$ be the diagonal matrix
whose first ${\rm rank}(\alpha)$ diagonal entries (from the upper left) are $1$, and the rest are $0$.  Then there
exists a unitary matrix $U(\alpha)$ such that (\cite[pg 195]{lantis}):
\begin{equation}
U(\alpha)P_{\alpha}U(\alpha)^{\dagger} = I({\rm rank}(\alpha)),
\end{equation}
If we apply the same similarity transformation to the representation $D_{\alpha}$, then it is easily seen that
\begin{equation}
U(\alpha) D_{\alpha}(h) U(\alpha)^{\dagger} = \left[ \oplus_{q=1}^{{\rm rank}(\alpha)} \tau(h) \right] \oplus D_{\alpha}^{H}(h), 
\quad h\in H.
\end{equation}
In the decomposition above, $\tau$ denotes the trivial representation\footnote{The trivial representation $\tau$ of a group $G$ is $\tau(g)=1$ for all $g\in G$. In Fourier terminology, the trivial representation gives the zero frequency of ``d.c.'' term.} of $H$, and the last term $D_{\alpha}^{H}$ is some
unitary representation of $H$ that does not contain $\tau$.  

Rather than starting with an arbitrary selection of $\{D_{\alpha}\}_{\alpha\in{\cal G}}$, suppose now that we 
choose one in which each matrix $D_{\alpha}(h)$ is exactly equal to a direct sum where the first ${\rm rank}(\alpha)$
representations that appear in the sum are $\tau$, i.e.,
\begin{equation}
D_{\alpha}(h) = \left[\oplus_{q=1}^{{\rm rank}(\alpha)} \tau(h)\right]\oplus D_{\alpha}^{H}(h),\quad
h\in H.
\label{eq:3p11}
\end{equation}
We always obtain such a {\em convenient selection} (that is what we shall call it henceforth) from a given
one by applying similarity transformations as described above. For a convenient selection, the projection
matrices in eq.~(\ref{eq:3p10}) are simply $P_{\alpha} = I({\rm rank}(\alpha))$ for all $\alpha$.  

\begin{lemma}
\label{lem:3p3p4} Let $\{D_{\alpha}\}_{\alpha \in {\cal G}}$ be a convenient selection and let $\{P_\alpha\}_{\alpha\in{\cal G}}$ be
its projections.  The nonzero coefficients in the matrices $\{ P_{\alpha} D_{\alpha} \}$ are precisely those coefficients
of the selection that are left $H$-invariant.  
\end{lemma}
The proof is given in the accompanying report \cite{kakaralaarxiv}.  

Since the left $H$-invariant coefficients are a basis for $\Theta_{H}(G)$, any linear map $\omega:\Theta_{H}(G)\rightarrow\field{C}$
is uniquely determined by the values that it gives to those coefficients.  For each matrix $P_{\alpha}D_{\alpha}$, the map $\omega$
produces a corresponding matrix $\eta(P_{\alpha} D_{\alpha})$. We now determine conditions in terms of the matrices $\eta(PD)$
under which $\omega$ is not only linear but also multiplicative and conjugate-preserving.  In the following, we use the standard
inner product $<\zeta_1,\zeta_2> = \zeta_1 \zeta_2^{\dagger}$ for complex-valued row vectors $\zeta_1$, $\zeta_2$, and the
standard norm $\| \zeta \| = (<\zeta,\zeta>)^{\frac{1}{2}}$.  

\begin{theorem}
\label{thm:3p3p5}
Let $\{D_\alpha\}_{\alpha\in{\cal G}}$ be a convenient selection and let $\{P_{\alpha}\}_{\alpha\in{\cal G}}$ be its projections.
Any linear map $\omega: \Theta_{H}(G)\rightarrow \field{C}$ is both multiplicative and conjugate-preserving if and only if
the following two conditions hold for all $\sigma$, $\delta$, $\alpha$ in ${\cal G}$:
\begin{eqnarray}
\eta(P_\sigma D_\sigma) \otimes \eta(P_\sigma D_\sigma) &=& 
\left[P_\sigma \otimes P_\delta\right] C_{\sigma\delta} \left[\eta(P_{\alpha_1}D_{\alpha_1})\oplus
\cdots \right.\nonumber\\
& & \oplus\eta(P_{\alpha_k}D_{\alpha_k})]C_{\sigma\delta}^{\dagger}; \label{eq:3p12}\\
\eta(P_{\alpha}D_{\alpha})\eta(P_{\alpha}D_{\alpha})^{\dagger} &=& P_{\alpha}.
\label{eq:3p13}
\end{eqnarray}
In eq.~(\ref{eq:3p12}), the matrix $C_{\sigma\delta}$ and the indices $\alpha_1$,\ldots,$\alpha_k$ are as in eq.~(\ref{eq:3p5}).  
\end{theorem}

The proof is given in the report \cite{kakaralaarxiv}.  

Let $f$ be a function in $L_{1}(G)$ such that $f(hg) = f(g)$ for all $h$ in a given closed subgroup $H$
of $G$.  The translation property of the Fourier transform ensures that each Fourier coefficient $F(\alpha)$
satisfies the identity $F(\alpha) = F(\alpha) D_{\alpha} (h)$ for all $h$ in $H$.  Integrating over $h$,
we find that $F(\alpha) = F(\alpha)P_{\alpha}$ for all $\alpha$.  We say that each Fourier coefficient 
$F(\alpha)$ is of {\em maximal $H$-rank} if the rank of $F(\alpha)$ equals the rank of $P_{\alpha}$.  
We now show that if $f$ is any left $H$-invariant function whose Fourier coefficients $F$ all have
maximal rank, then $f$ is uniquely determined by its bispectrum $A_{3,f}$ up to a left
translation.  The proof of our assertion uses the standard notation from linear algebra \cite{lantis}. 
For each matrix $A$, let ${\rm image}(A)$ and ${\rm ker}(A)$ denote respectively the image and
kernel of $A$.  For each $\alpha\in{\cal G}$, let ${\cal H}_{\alpha}$ denote the Hilbert space on
which the corresponding representations $D_{\alpha}$ act.

\begin{theorem}
\label{thm:3p3p6}
Let $G$ be any compact group, and let $H$ be any closed subgroup of $G$. Let $r\in L_{1}(G)$ be invariant
under left $H$-translations.  If the Fourier coefficients $\{R(\alpha)\}_{\alpha\in{\cal G}}$ all have
maximal $H$-rank, then $a_{3,r} = a_{3,s}$ for some $s\in L_{1}(G)$ if and only if there exists
$x\in G$ such that $s(g) = r(xg)$ for all $g$.  
\end{theorem}

The proof is given in the report \cite{kakaralaarxiv}.

In the theorem above, we did not require that the function $s$ also be left $H$-invariant. (Equality of bispectra may hold regardless of whether both functions are $H$-invariant.) Suppose now that two left $H$-invariant
functions $r$, $s$ are such that both have maximal $H$-rank coefficients and both have exactly the same bispectrum.  Theorem~\ref{thm:3p3p6} demonstrates that under those conditions, there exits $x\in G$
such that $s(g) = r(xg)$ for all $g$.  Yet the element $x$ cannot be arbitrary, for $s$ is left $H$-invariant,
and thus $s(hg)=s(g)$, implying that $r(xhg)=r(xg)$ for all $h\in H$ and $g\in G$.  But since $r$ is also
left $H$-invariant, we must have $r(xg) = r(hxg)$, and thus $r(xhg)=r(hxg)$ for all $g$ and $h$. The last identity
is always satisfied if $x$ lies in the {\em normalizer} of $H$ in $G$, which is the subgroup $N_{H}$ of $G$
defined as follows:
\begin{equation}
N_{H} = \{ x\in G: xH = Hx\}.
\end{equation}
(The normalizer of $H$ is the largest subgroup $N_{H}$ of $G$ such that $G$ itself is a normal subgroup of $N_{H}$.)
In fact, we show that $x$ {\em must} lie in $N_{H}$ in the following theorem.

\begin{theorem}
\label{thm:3p3p7}
Let $r$, $s$ in $L_{1}(G)$ be two left $H$-invariant functions whose Fourier coefficients $R(\alpha)$ and $S(\alpha)$ 
both have maximal $H$-rank for all $\alpha$. Then $a_{3,r} = a_{3,s}$ if and only if $s(g)=r(xg)$ for some $x\in N_{H}$.
\end{theorem}
\begin{proof}
The ``if'' assertion is shown above, so we prove the ``only if'' part.  Suppose that $a_{3,r}=a_{3,s}$, and that $r$, $s$,
both have maximal $H$-rank coefficients.  Under those conditions, Theorem \ref{thm:3p3p6} shows that there exists $x\in G$
such that $r(g) = s(xg)$ for all $g$.  Then $R(\alpha)=S(\alpha)D_{\alpha}(x)$ for all $\alpha\in{\cal G}$.  Furthermore, the left
invariance of $r$ implies that $R(\alpha)=R(\alpha)P_{\alpha}$ for each $\alpha$,  Thus 
$S(\alpha)D_{\alpha}(x) = S(\alpha)D_{\alpha}(x)P_{\alpha}$ for each $\alpha$, and combining that with the identity
$S(\alpha) = S(\alpha)P_{\alpha}$ yields $S(\alpha)P_{\alpha}D_{\alpha}(x) = S(\alpha)P_{\alpha}D_{\alpha}(x)P_{\alpha}$,
and thus $S(\alpha)\left[P_{\alpha}D_{\alpha}(x) - P_{\alpha}D_{\alpha}(x)P_{\alpha}\right] = 0$.  By the maximal 
$H$-rank hypothesis, we obtain that 
\begin{equation}
P_{\alpha}D_{\alpha}(x) = P_{\alpha}D_{\alpha}(x)P_{\alpha}.
\end{equation}
Since $P_{\alpha}=I({\rm rank}(\alpha))$ for a convenient selection, the element $x$ satisfies the above equality
if and only if the unitary matrix $D_{\alpha}(x)$ is the direct sum of two smaller unitary matrices, the first with
dimensions ${\rm rank}(\alpha)\times {\rm rank}(\alpha)$ and the second with dimensions
$(n-{\rm rank}(\alpha))\times (n-{\rm rank}(\alpha))$. For such an $x$, it follows for any $h\in H$ that
\begin{equation}
P_{\alpha}D_{\alpha}(x)^{\dagger} D_{\alpha}(h) D_{\alpha}(x) = P_{\alpha} D_{\alpha}(x^{-1} h x) = P_{\alpha}.
\end{equation}
But we now see by Lemma \ref{lem:3p3p4} that $P_{\alpha}D_{\alpha}(x^{-1} h x) = P_{\alpha}$ if and only if 
$x^{-1} h x \in H$. The last inclusion holds for all $h\in H$, and thus $x^{-1} H x = H$, or equivalently,
$x\in N_{H}$. \qquad\end{proof}

In the interesting special case when $G$ is the group $SO(3)$ and $H$ is the subgroup of rotations that fix the 
$z$-axis, we have that $N_{H}=H$.  In that case, if $r$ is any left $H$-invariant function with maximal
$H$-rank coefficients, then there are no other left $H$-invariant functions with the same bispectrum
besides $r$ itself.  However, that does not mean that the bispectrum uniquely determines $r$: any function
$s$ such that $s(g) = r(xg)$ on $G$ has the same bispectrum, although $s$ is not necessarily $H$-invariant. 

If $G=SO(3)$ and $H$ as above, then the maximal $H$-rank condition is easy to satisfy.  Here it is well-known
that ${\rank}(P_{\alpha}) = 1$ for all $\alpha \in {\cal G}$ (\cite{varshalovich}).  Thus an arbitrary left 
$H$-invariant function $r$ has maximal $H$-rank coefficients if for all $\alpha$, the matrix $R(\alpha)$ contains
at least one nonzero coefficient.  That is evidently true if any noise is present in measuring $r$.

\section{Reconstruction algorithms}
\label{sec:algos}

The completeness theory for arbitrary compact groups in the preceding sections can be refined further for the special
case when the group is the $3$-D rotation group $SO(3)$.  We provide in this section a constructive algorithm for recovering a function from its bispectrum on $SO(3)$ . The algorithm was first outlined in an earlier conference paper \cite{kakarala93}.  It  relies on the bispectrum formula (\ref{eq:biss03}), in which the indices $\sigma$, $\delta$ to the bispectrum $A_{3,f}$ are nonnegative integers.  Setting $\delta=1$, we obtain from (\ref{eq:biss03}) insight into how a recursive algorithm may be built:
\begin{eqnarray}
A_{3,f}(\ell-1,1) &=& F(\ell-1)\otimes F(1) C_{\ell-1,1}\left[ F(\ell) \oplus \right.\nonumber\\
& & F(\ell-1) \oplus F(\ell-2) ] C_{\ell-1,1}^{\dagger}
\label{eq:recurse}
\end{eqnarray}
The above suggests that if we know $F(1)$, $\ldots$, $F(\ell-1)$, then we may recover $F(\ell)$.  We take advantage of that insight in the following algorithm

The algorithm requires the assumption that $f: SO(3)\rightarrow \field{R}$ has nonsingular Fourier coefficients $F(\ell)$ for $\ell=0,\ldots, L$.   With that assumption, it proceeds as follows:
\begin{enumerate}
\item By (\ref{eq:biss03}), we have that $A_{3,f}(0,0)=F(0)^{3}$ for the real number $F(0)$, and consequently, 
\beq
F(0) = \sqrt[3]{A_{3,f}(0,0)}
\eeq
\item We estimate $F(1)$.  From (\ref{eq:biss03}), we have that 
\beq
\frac{A_{3,f}(1,0)}{F(0)} = F(1)F(1)^{\dagger}
\eeq
Let $\hat{F}(1)$ denote the square root of the positive definite matrix on the right.  We know that $\hat{F}(1)$ is also positive definite  \cite[pg 190]{lantis},  In the Appendix, it is shown that there exists $g\in SO(3)$ such that $\hat{F}(1) = F(1)D_{1}(g)$.  
\item If $L=1$, then we are done.  Otherwise, we employ the recursion (\ref{eq:recurse}).  Since we know $\hat{F}(1)$ and $A_{3,f}(1,1)$, we obtain $\hat{F}(2)$ from 
\beq
C_{11}^{\dagger}\left[ \hat{F}(1) \otimes \hat{F}(1) \right]^{-1} A_{3,f}(1,1) C_{11}
\eeq
Note that the inverse is possible because $\hat{F}$ is nonsingular and so is the Kronecker product in brackets.  After substituting $\hat{F}(1)=F(1)D_1(g)$, we obtain that the above simplifies to 
\beq
\left[ D_{2}(g)^{\dagger} F(2)^{\dagger}\right] \oplus  \left[ D_{1}(g)^{\dagger} F(1)^{\dagger}\right] \oplus F(0).
\eeq
Since $F(2)$ is $2\cdot 2 +1 = 5$ dimensional, we denote the upper left $5\times 5$ submatrix of the above $9\time 9$ matrix as $\hat{F}(2)$.  Note that the value of $g$ remains unknown at this point, but it is the same value of $g$ for both $\hat{F}(2)$ and $\hat{F}(1)$.  
\item For $\ell\leq L$, we obtain $\hat{F}(\ell)$ using the equation
\beq
C_{\ell-1,1}^{\dagger}\left[ \hat{F}(\ell-1) \otimes \hat{F}(1) \right]^{-1} A_{3,f}(\ell-1,1) C_{\ell-1,1}
\eeq
After inserting (\ref{eq:recurse}) and simplifying, we get that the upper left $2\ell+1$-dimensional submatrix is $\hat{F}(\ell)=F(\ell)D_{\ell}(g)$.
\end{enumerate}

The algorithm allows us to state the following, which extends Theorem~\ref{thm:3p2p4} to the bandlimited case.  

\begin{theorem}
\label{thm:bandlimited}
Let $L>0$, suppose that $r$ on $SO(3)$ has nonsingular Fourier coefficients $R(\ell)$ for $\ell \leq L$ and $R(\ell)=0$ for $\ell > L$, then $a_{3,r}=a_{3,s}$ if and only if there exists $x\in G$ such that $r(g)=s(xg)$ for all $g\in G$.
\end{theorem}

We should note that by Lemma~\ref{lem:bisequali2}, the same theorem holds if we replace $a_{3,r}$ by the second order invariants $I^{2}_r$.  The advantage of using the bispectrum in place of $I^{2}_r$ is that it helps us to see the recursive structure (\ref{eq:recurse}) clearly.  

The algorithm serves two purposes: to prove Theorem~\ref{thm:bandlimited}, and to show a constructive method to recover a signal from its bispectrum, which is more illuminating than an existence proof of recovery as used in Theorem 2. The practicality of the algorithm is not known; whether the recursive nature of the algorithm or the repeated use of matrix inverse lead to numerical instability would be an interesting question to explore.  It is worth noting that similar recursive algorithms have been devised when $G=\field{R}$, and shown to be practically useful, not only theoretically interesting \cite{bartelt}.

\section{Summary and future directions}

This paper derives completeness properties of the bispectrum for functions defined on compact groups and their homogeneous spaces. A matrix form of the bispectrum is derived, and it is shown that every function with nonsingular coefficients is completely determined, up to a group translation, by its bispectrum.  A reconstruction algorithm for functions defined on the groups $SU(2)$ and $SO(3)$ is described.  The main theoretical result shows that the bispectrum is a complete source of invariants for homogeneous spaces of compact groups.

Results similar to those in this paper may be established for non-compact, non-commutative groups.  In the author's Ph.d. thesis \cite{kakaralathesis},
the completeness of the bispectrum for locally compact groups is established using the duality theorem of Tatsuuma. Those results will be reported
in a subsequent paper.  The Tannaka-Krein duality theorem, which is
central to this paper, has recently been extended to compact groupoids \cite{amini}. It would be interesting to see if a 
corresponding bispectral theory may be constructed there.

\begin{acknowledgements}
 I thank the numerous people who wrote for a copy of my Ph.D. dissertation \cite{kakaralathesis}, in which
this work was first presented.  I also thank the anonymous reviewers for their comments.  On writing this, I realize how much my late PhD supervisor, Professor Bruce M. Bennett, contributed to my studies and research.  His passion for mathematics, deep insights into abstract concepts, and guidance in life decisions continues to benefit me today.  
\end{acknowledgements}

\section*{Appendix}
\appendix

We prove that $\hat{F}(1) = F(1)D_{1}(g)$.  The representation $D_1$ of $SO(3)$ is such that, for some fixed unitary $U$, we have $D_{1}(g) = UgU^{\dagger}$ for all $g$  in $SO(3)$.  Thus
\begin{eqnarray*}
F(1) &=& \int_{G} f(g) D_{1} (g)^{\dagger} dg,\\
&=& U \left[ \int_{G} f(g) g^{\dagger} dg \right] U^{\dagger}.
\end{eqnarray*}
Let $F_{s}(1)$ denote the matrix that results by evaluating the integral in brackets.  Since $f$ is real-valued,
and every matrix $g$ has real coefficients, the matrix $F_{s}(1)$ has only real coefficients.  Thus the determinant
of $F(1) = U F_{s}(1) U^{\dagger}$ is a real number.  Assume for the moment that 
${\rm det}\left[ F(1) \right] = {\rm det}\left[F_{s}(1)\right] > 0$. Let $\hat{F}(1)$ and $\hat{F}_{s}(1)$ denote
respectively the (unique) positive square roots of $F(1)F(1)^{\dagger}$ and $F_{s}(1)F_{s}(1)^{\dagger}$.  Since 
$F(1)F(1)^{\dagger} = U F_{s}(1)F_{s}(1)^{\dagger} U^{\dagger}$, it is easily seen that $\hat{F}(1) = U \hat{F}_s(1) U^{\dagger}$.
Now consider the polar decomposition $F_{s}(1) = HV$, where $H$ is positive definite and $V$ is unitary.  Note that 
$H = \left(F_{s}(1) F_{s}(1)^{\dagger} \right)_{+}^{\frac{1}{2}}$, and thus
$H = \hat{F}_{s}(1)$.  Since $F_{s}(1)$ is real-valued, $V$ must be real-valued orthogonal matrix.  Matching 
determinants on both sides of the equation $F_{s}(1) = \hat{F}_{s}(1) V$ reveals that ${\rm det}[V] = +1$, and
thus $V = g$, for some $g \in SO(3)$. Substitution reveals that 
\begin{eqnarray*}
\hat{F}(1) &=& U\hat{F}_{s}(1) U^{\dagger} = U F_{s}(1) g U^{\dagger} \\
&=& U F_{s}(1) U^{\dagger} U g U ^{\dagger} = F(1) D_{1}(g).
\end{eqnarray*}

The assumption that ${\rm det}[F(1)] > 0$ is not critical.  We use it only to obtain that ${\rm det}[V] = +1$,
where $V = \hat{F}_{s}(1)^{-1} F_{s}(1)$.  Instead of selecting $\hat{F}(1)$ to be the positive definite square
root of $F(1)F(1)^{\dagger}$, we may choose $\hat{F}(1)$ to be any square root such that 
${\rm det}[\hat{F}(1)] = {\rm det}[F(1)]$, e.g., by multiplying the top row of the positive definite 
square root matrix by $-1$ if necessary.  We do not know ${\rm det}[F(1)]$ {\em a priori}, but if we
store it as ``side information'' along with the bispectrum, then we obtain a complete rotation-invariant
description for any real-valued bandlimited function on $SO(3)$.  Note that ${\rm det}[F(1)]$ remains invariant
under translation on $SO(3)$, i.e., if $f(g) = s(hg)$, then $F(1) = S(1)D_{1}(h)$, but since ${\rm det}[D_{1}(h)] = +1$,
we obtain that ${\rm det}[F(1)] = {\rm det}[S(1)]$.


\begin{thebibliography}{10}
\bibitem{amini}
M.~Amini.
\newblock Tannaka-Krein duality for compact groupoids {I}: representation
  theory.
\newblock {\em Advances in mathematics}, 214:78--91, 2007.

\bibitem{barut}
A.~O. Barut and R.~Raczka.
\newblock {\em Theory of group representations and applications, 2nd Ed}.
\newblock World Scientific, Singapore: 1986.

\bibitem{brillinger}
D.~Brillinger.
\newblock Some history of higher-order statistics and spectra.
\newblock {\it Statistica Sinica} 1, 465-476, 1991. 
 
\bibitem{bartelt}
H.~Bartelt, A.~W.~Lohmann, and B.~Wirnitzer.
\newblock Phase and amplitude recovery from bispectra.
\newblock {\em Applied optics}, 23(18): 3121-3129, 1984 

\bibitem{candesphase}
E.J.~Cand\`{e}s, T.~Strohmer and V. Voroninski. 
\newblock PhaseLift: exact and stable signal recovery from magnitude measurements via convex programming. 
\newblock To appear in {\em Communications on pure and applied mathematics}, 2011.

\bibitem{chaichian}
M.~Chaichian and R.~Hagedorn.
\newblock {\em Symmetries in quantum mechanics: from angular momentum to supersymmetry}.
\newblock Taylor\& Francis, London: 1998.

\bibitem{chevalley}
C.~Chevalley.
\newblock {\em Theory of Lie groups}.
\newblock Princeton University Press, Princeton, NJ: 1946.

\bibitem{chung}
M.~K.~Chung, K.~M.~Dalton, L.~Shen, A.~C.~Evans, R.~J.~Davidson
\newblock Weighted Fourier series representation and its application to quantifying the amount of gray matter.
\newblock {\em IEEE Transactions on medical imaging} 26(4): 566-581, 2007.

\bibitem{courant-hilbert}
R.~Courant and D.~Hilbert.
\newblock {\em Methods of mathematical physics, pt I}.
\newblock John Wiley, New York: 1989.

\bibitem{diaconismono}
P.~Diaconis.
\newblock {\em Group representations in probability and statistics}.
\newblock Institute of Mathematical Statistics, Hayward, CA: 1988.

\bibitem{Fehr10}
J.~Fehr.
\newblock Local rotation invariant patch descriptors for 3D vector fields.
\newblock In {\em Proceedings of the 2010 International conference on pattern recognition (ICPR)}, pp 1381-1384, 2010.

\bibitem{FlusserBZ03}
J.~Flusser, J.~Boldys, and B.~Zitov{\'a}.
\newblock Moment forms invariant to rotation and blur in arbitrary number of dimensions.
\newblock {\em IEEE Transactions on pattern analysis and machine intelligence}, 25(2):234-246, 2003.

\bibitem{flusserbook}
J.~Flusser, T.~Suk, and B.~Zitov{\'a}.
\newblock {\em Moments and moment invariants in pattern recognition}.
\newblock John Wiley, New York: 2009.

\bibitem{FromeHKBM04}
A.~Frome, D.~Huber, R.~Kolluri, T.~B{\"u}low, and J.~Malik.
\newblock Recognizing objects in range data using regional point descriptors.
\newblock {\em Proceedings of the 2004 European conference on computer vision (ECCV)}, pp 224--237, 2004.

\bibitem{GalvezC93}
J.~M. Galvez and M.~Canton.
\newblock Normalization and shape recognition of three-dimensional objects by 3d moments.
\newblock {\em Pattern recognition}, 26(5):667-681, 1993.

\bibitem{gamo}
H.~Gamo.
\newblock Triple correlator of photoelectric fluctuations as a spectroscopic tool.
\newblock {\em Journal of applied physics}, 34:875-876, 1963.

\bibitem{gauthier}
J.~P. Gauthier, G.~Bornard, and M.~Silbermann.
\newblock Motions and pattern analysis - harmonic analysis on motion groups and their homogeneous spaces.
\newblock {\em IEEE Transactions on systems, man, and cybernetics}, 21(1):159-172, 1991.

\bibitem{Giannakis89}
G.~B. Giannakis.
\newblock Signal reconstruction from multiple correlations: frequency- and  time-domain approaches.
\newblock {\em Journal of the Optical Society of America A}, 6(5):682-697, 1989.

\bibitem{guth}
A.~H. Guth.
\newblock {\em The inflationary universe: the quest for a new theory of cosmic origins}.
\newblock Perseus, Cambridge, MA: 1997.

\bibitem{hamermesh}
M.~Hamermesh.
\newblock {\em Group theory and its application to physical problems}.
\newblock Dover, New York: 1962.

\bibitem{healyrockmore:improvedff2}
D.~M. Healy~Jr., D.~Rockmore, P.~Kostelec, and S.~S.~B. Moore.
\newblock Ffts for the 2-sphere--improvements and variations.
\newblock {\em The journal of Fourier analysis and applications},
  9(4):341--385, 2003.

\bibitem{hewittross}
E.~A. Hewitt and K.~A. Ross.
\newblock {\em Abstract harmonic analysis: Volume 2: Structure and analysis for compact groups. Analysis on locally compact abelian groups}.
\newblock Springer-Verlag, New York: 1970.

\bibitem{Hu}
M.~K. Hu.
\newblock Visual pattern recognition by moment invariants.
\newblock {\em IRE Transactions on information theory}, vol. IT-8, pp.179-187, 1962.

\bibitem{iwahorisugiura}
N.~Iwahori and M.~Sugiura.
\newblock A duality theorem for homogeneous manifolds of compact lie groups.
\newblock {\em Osaka journal of mathematics}, 3:139-153, 1966.

\bibitem{kanatani}
K.~Kanatani.
\newblock {\em Group theoretical methods in image understanding.}
\newblock Springer-Verlag, New York: 1990.

\bibitem{kakaralaarxiv}
R.~Kakarala.
\newblock Completeness of bispectrum on compact groups.
\newblock arXiv:0902.0196v1, 2009.

\bibitem{kakarala93}
R.~Kakarala, B.~M. Bennett, G.~J. Iverson, and M.~D'Zmura.
\newblock Bispectral techniques for spherical functions.
\newblock {\em Proceedings of the 1993 International conference on acoustics, speech, and signal processing (ICASSP)}, Vol.~4, pp 216-219, 1993.

\bibitem{kakaralacvpr2010}
R.~Kakarala and D.~Mao.
\newblock A theory of phase-sensitive rotation invariance with spherical
  harmonic and moment-based representations.
\newblock {\em Proceedings of the 2010 IEEE conference on computer vision and
  pattern recognition (CVPR)} pp 105-112, 2010.
  
\bibitem{kakaralahau3d}
R.~Kakarala, P.~Kaliamoorthi, and W.~Li.
\newblock  Viewpoint invariants from three-dimensional data: the role of reflection in human activity understanding.
\newblock {\em Proc. 2011 CVPR workshop on human activity understanding from 3D data (HAU3D)}, pp 57-62, 2011.

\bibitem{kakaralathesis}
R.~Kakarala.
\newblock {\em Triple correlation on groups}.
\newblock PhD thesis, University of California, Irvine, 1992.

\bibitem{kakaralatsp}
R.~Kakarala.
\newblock A signal processing approach to the Fourier analysis of ranking data: the importance of phase.
\newblock {\em IEEE Transactions on signal processing}, 59(4):1518-1527, 2011.

\bibitem{kanatani84}
K.~Kanatani.
\newblock Distribution of directional data and fabric tensors.
\newblock {\em International journal of engineering science}, 22(2):149-164,   1984.

\bibitem{KazhdanFR03}
M.~M. Kazhdan, T.~A. Funkhouser, and S.~Rusinkiewicz.
\newblock Rotation invariant spherical harmonic representation of 3d shape descriptors.
\newblock {\em Symposium on geometry processing}, pp 156-165, 2003.

\bibitem{kondor:icml08}
R.~Kondor and K.~Borgwardt.
\newblock The skew spectrum of graphs.
\newblock In A.~McCallum and S.~Roweis, editors, {\em Proceedings of the
  2008 International conference on machine learning}, pp 496-503, 2008.

\bibitem{kondorthesis}
R.~Kondor.
\newblock {\em Group theoretical methods in machine learning}.
\newblock PhD thesis, Columbia University, 2008.

\bibitem{Kostelec}
P.~Kostelec and D.~Rockmore.
\newblock FFTs on the rotation group.
\newblock {\em Journal of Fourier analysis and applications}, 14(2):145-179, 2008.

\bibitem{kyatkin}
A.B.~Kyatkin and G.S.~Chirikjian.
\newblock Algorithms for fast convolutions on motion groups.
\newblock {\em Applied and computational harmonic analysis} 9(2):220-241, 2000. 

\bibitem{lantis}
P.~Lancaster and M.~Tismenetsky.
\newblock {\em The theory of matrices}, 2nd Ed.
\newblock Academic Press, San Diego: 1985.

\bibitem{Leduc}
 J.-P.~Leduc.
 \newblock A group-theoretic construction with spatiotemporal wavelets for the analysis of rotational motion.
 \newblock {\em Journal of mathematical imaging and vision}, 17(3):207-236, 2002.

\bibitem{Lo89}
C.-H. Lo and H.-S. Don.
\newblock 3-d moment forms: Their construction and application to object
  identification and positioning.
\newblock {\em IEEE Transactions on pattern analysis and machine intelligence}, 11(10):1053-1064, 1989.

\bibitem{luo}
X.~Luo.
\newblock The angular bispectrum of the cosmic microwave background.
\newblock {\em Astrophysical journal}, 427:L71--L74, 1994.

\bibitem{mujica}
F.A.~Mujica, J.-P.~Leduc, M.J.T.~Smith and R.~Murenzi.
\newblock Spatiotemporal wavelets: A group-theoretic construction for motion estimation and tracking.
\newblock {\em SIAM Journal of applied mathematics} 61(2): 596-632, 2000. 

\bibitem{muller1997}
R.~A. Muller and G.~J. Macdonald.
\newblock Spectrum of 100-kyr glacial cycle: Orbital inclination, not eccentricity.
\newblock {\em Proceedings of the National Academy of Sciences}, 94:8329-8334, 1997.
  
\bibitem{murenzi}
R.~Murenzi.
\newblock Wavelet transforms associated to the n-dimensional Euclidean group with dilations: signals in more than one dimension.
\newblock In {\em Wavelets: time-frequency methods and phase space}, (J.M. Combes, A. Grossmann, Ph. Tchamitchian, eds), 
\newblock pp 239-246.  Springer-Verlag, New York: 1990.

\bibitem{naimark}
M.~A. Naimark and A.~I. Stern.
\newblock {\em Theory of group representations}.
\newblock Springer-Verlag, New York, 1982.

\bibitem{ReisertB06}
M.~Reisert and H.~Burkhardt.
\newblock Using irreducible group representations for invariant 3d shape description.
\newblock {\em Proc. of 2006 DAGM Symposium}, pp. 132-141, 2006.

\bibitem{Sadjadi80}
F.~A. Sadjadi and E.~L. Hall.
\newblock Three dimensional moment invariants.
\newblock {\em IEEE Transactions on pattern analysis and machine intelligence}, 2(2):127-136, 1980.

\bibitem{CosmoBis06}
E.~Sefusatti, M.~Crocce, S.~Pueblas and R. Scoccimarro.
\newblock Cosmology and the bispectrum.
\newblock {\em Physical review D}, 74 (2): 23522-23546, 2006.

\bibitem{Shilane04}
P.~Shilane, P.~Min, M.~Kazhdan, and T.~Funkhouser.
\newblock The {P}rinceton shape benchmark.
\newblock In {\em Shape modeling international (SMI)}, pp. 167--178, 2004.

\bibitem{edwards}
R.~E. Edwards.
\newblock {\em Integration and harmonic analysis on groups}.
\newblock Cambridge University Press, Cambridge, MA: 1972.





%



\bibitem{smach}
F.~Smach, C.~Lemaitre, J.~P. Gauthier, J.~Miteran, and M.~Atri.
\newblock Generalized Fourier descriptors with applications to object matching
  in the svm context.
\newblock {\em Journal of mathematical imaging and vision}, 30(1):43-71, 2008.

\bibitem{sugiura}
H.~Sugiura.
\newblock {\em Unitary group representations and harmonic analysis}.
\newblock Halsted Press, New York: 1975.

\bibitem{varshalovich}
D.~A. Varshalovich, A.~N. Moskalev, and V.~K. Kershonskii.
\newblock {\em Quantum theory of angular momentum}.
\newblock World Scientific, Singapore: 1988.

\bibitem{yellottiverson}
J.~I. Yellott~Jr. and G.~J. Iverson.
\newblock Uniqueness theorems for generalized autocorrelation functions.
\newblock {\em Journal of the Optical Society of America A}, 9(3):388--401, 1992.

\bibitem{zelobenko}
D.~P. Zelobenko.
\newblock {\em Compact Lie groups and their representations}.
\newblock American Mathematical Society, Providence, RI: 1973.

\end{thebibliography}


\end{document}